\documentclass[11pt]{amsart}
\usepackage{geometry} 
\usepackage{pgfplots}
\usepackage{psfrag}
\usepackage{amsmath}
\usepackage{calc}
\usepackage{systeme}
\usepackage{amsfonts}
\usepackage{enumerate}
\usepackage{hyperref}

\usepackage{amsthm}
\usepackage{algorithm2e}
\usepackage[applemac]{inputenc}
\usepackage{bbold}
\usepackage[numbers]{natbib}

\newcommand{\R}{\mathbb{R}}
\newcommand{\be}{\begin{equation}}
\newcommand{\ee}{\end{equation}}
\newcommand{\ba}{\begin{aligned}}
\newcommand{\ea}{\end{aligned}}
\newcommand{\eps}{\epsilon}
\newcommand{\mpt}{\mathcal{P}_2(\R^d)}
\newcommand{\mpp}{\mathcal{P}_p(\R^d)}

\newcommand{\mpo}{\mathcal P(\mathcal O)}
\newcommand{\mo}{\mathcal O}
\newcommand{\mm}{\mathcal P(M)}

\newtheorem{Theorem}{Theorem}[section]

\newtheorem{Rem}[Theorem]{Remark}
\newtheorem{Def}[Theorem]{Definition}
\newtheorem{Prop}[Theorem]{Proposition}

\newtheorem{Ex}[Theorem]{Example}

\newtheorem{Def/Prop}[Theorem]{Definition/Proposition}

\newcommand{\triplenorm}[1]{{\left\vert\kern-0.25ex\left\vert\kern-0.25ex\left\vert #1 
    \right\vert\kern-0.25ex\right\vert\kern-0.25ex\right\vert}}

\title{The tangent space to the Wasserstein space: parallel transport and other applications}

\begin{document}





\author{Charles Bertucci \textsuperscript{a}}
\address{ \textsuperscript{a} CEREMADE, CNRS, Universit\'e Paris Dauphine-PSL, UMR 7534, 75016 Paris, France.}

\maketitle

\begin{abstract}
We propose a new notion of the \emph{formal} tangent space to the Wasserstein space $\mathcal P(X)$ at a given measure. Modulo an integrability condition, we say that this tangent space is made of functions over $X$ which are valued in the probability measures over the tangent bundle to X. This generalization of previous concepts of tangent spaces allows us to define appropriate notions of parallel transport, $\mathcal C^{1,\alpha}$ regularity over $\mathcal P(X)$ and translation of a curve over $\mathcal P(X)$.
\end{abstract}


\setcounter{tocdepth}{1}

\tableofcontents
\section*{Introduction}
We introduce a new notion of tangent space to the space of probability measures over different finite dimensional metric sets, when it is equipped with a Wasserstein metric (i.e. a Wasserstein space). This notion is more general than the one previously considered in the literature in the sense that it contains strictly more elements. Our notion basically sees elements in the tangent space as measured valued functions, instead of simply functions, in the same spirit as the famous relaxations of either Young or Kantorovich.\\

We then proceed by using this point of view to define a notion of parallel transport on the space of probability measures (Section \ref{sec:para}). Later, we build on this notion of parallel transport to define $\mathcal C^{1,\alpha}$ spaces of real valued functions over the space of probability measures equipped with Wasserstein metrics (Section \ref{sec:c1a}). We finally use a similar point of view to explain how we can translate curves and how it is helpful in optimal control (Section \ref{sec:trans}).\\

We feel the need to insist upon the fact that, despite the maybe poor use of the differential geometry vocabulary in what follows (tangent space, parallel transport,...), we do not claim that the Wasserstein space is a smooth Riemannian manifold onto which such concepts are well defined. Nonetheless, we claim that adopting such a geometric point of view leads to the definition of tools that are useful in the analysis of several problems.\\

More detailed bibliographical comments are given in the core of the paper, but we now present briefly the context. Since the seminal work of Otto \citep{otto} and the influential book of Ambrosio, Gigli and Savar\'e \citep{ags}, Wasserstein spaces are often thought of equipped with a formal Riemannian structure. The tangent space to the Wasserstein space $\mathcal P(X)$ is widely taken as the closure of gradients of smooth functions over the metric space $X$, closure which is taken for some metric which depends on the probability measure at which we are taking the tangent space. This point of view is mostly induced by an Eulerian description of flows on $\mathcal P(X)$. It has been extremely fruitful. It also allowed to define parallel transports along flows which can be written in an Eulerian way with sufficient regularity \citep{gigli,gigli2}. We provide here a more Lagrangian point of view which allows us to define parallel transport quite generally, along Lagrangian flows. We then present some natural consequences of this (Lagrangian) notion of parallel transport. Our point of view is clearly in line with the lift of Lions on the space of random variables introduced in \citep{lions}, which we generalize to more involved settings, see Section \ref{sec:lift}.

\section{Notation and standard results}
Let $d \geq 1$ be an integer. For a topological space $O$, $\mathcal P(O)$ stands for the set of Borel probability measures on $O$. For $p \in [1,\infty)$, we consider the space $\mpp:= \{m \in \mathcal P(\R^d)| \int_{\R^d}|x|^pm(dx) < \infty\}$ equipped with the $p$-Wasserstein distance $W_p$ defined by
$$
W_p(µ,\nu) = \left(\inf_{\gamma \in \Pi(µ,\nu)}\int_{\R^{2d}}|x-y|^p\gamma(dx,dy)\right)^{\frac1p},
$$
where $\Pi(µ,\nu)$ is the set of couplings between $µ$ and $\nu$, i.e. the set of measures $\gamma \in \mathcal P(\R^{2d})$ such that $(\pi_1)_\#\gamma =µ$ and $(\pi_2)_\#\gamma= \nu$, where, for $x,y \in \R^d$, $\pi_1(x,y) = x$ and $\pi_2(x,y) = y$, and $f_\#m$ denotes the image of the measure $m$ by the map $f$. The maps $\pi_1$ and $\pi_2$ shall be used more generally on any couple as the natural projection maps. The set $(\mpp,W_p)$ is a complete separable metric space.

Below, $\mathcal O$ shall be the closure of a smooth open bounded subset of $\R^d$. For $ p \in [1,\infty)$, $(\mathcal P(O),W_p)$ is a complete separable compact metric space.

For $p \in (1,\infty)$, we denote $p' = p/(p-1)$. We then define $\mathcal P_{pp'}(\R^{d}\times \R^d)= \{µ \in \mathcal P(\R^{2d}) | (\pi_1)_\# µ \in \mpp, (\pi_2)_\#µ \in \mathcal P_{p'}(\R^d)\}$. This set can be equipped with a sort of $(p,p')$-Wasserstein distance (see Section 10.3 in \citep{ags}), that we do not detail here.

Below, $M$ shall denote a smooth bounded manifold without boundary. Given $x \in M$, $T_xM$ is the tangent space to $M$ at $x$ and $TM:= \cup_{x \in M}\{x\}\times T_xM$ is the tangent bundle of $M$, which shall also be identified sometimes with $\cup_{x \in M}T_xM$. Given a smooth curve $\theta : [0,1] \to M$, and $p_0 \in T_{\theta(0)}M$, the parallel transport of $p_0$ along $\theta$ is characterized as the unique curve $p:[0,1] \to TM$ such that $\nabla_{\theta'(t)}p(t) = 0$ together with $p(0) = p_0$. The set $(\mm,W_p)$ is a complete separable compact metric space for any $p \in [1,\infty)$, where $W_p$ is defined by replacing $|x-y|$ by the geodesic distance between $x$ and $y$ in the definition above.

\section{Tangent space to $\mpp$}
\subsection{Main definitions}
The study of evolution equations on $\mpp$ lead to the need of working with the tangent space to $\mpp$ at some point $µ \in \mpp$. An early and influential work in this direction is \citep{otto} which was in fact concerned with $L^1$ spaces. Following \citep{ags} Section 8.4, several works including namely \citep{gangbo,carmona} have then argue that the tangent space to $(\mpp,W_p)$ at $µ \in \mpp$ is
 \be\label{bad}\overline {\nabla_x C^{\infty}(\R^d,\R)}^{L^{p'}((\R^d,µ),\R^d)}.\ee

Here, we present a different point of view and consider instead the tangent space to $\mpp$ at $µ \in \mpp$ to be given by
\be\label{tangent}
T_µ(\mpp) := \left\{ \psi : \R^d \to \mathcal P(\R^d), \text{ measurable } \bigg| \int_{\R^{2d}}|z|^{p'}\psi_x(dz)µ(dx) < \infty \right\}.
\ee
By disintegration, it is equivalent to considering the set of probability measures $\gamma \in \mathcal P_{pp'}(\R^{2d})$ such that $(\pi_1)_\#\gamma = µ$ and we shall sometimes say with a slight abuse of notation that some $\gamma(dx,dz) = µ(dx)\psi_x(dz)$ belongs to $T_µ(\mpp)$ if $\psi$ does. Pursuing in this direction, we define $T\mpp:= \cup_{µ \in \mpp}\{µ\}\times T_µ\mpp$ which will be identified with $\mathcal P_{pp'}(\R^{2d})$.

Of course, for all elements $\phi$ in the tangent space defined in \eqref{bad}, we have an associated elements in $T_µ(\mpp)$ by considering the map $\psi : x \to \delta_{\phi(x)}$, such elements shall be called deterministic.

We provide in the next sections applications which we believe justify precisely this notion of tangent space. We finish this section by giving some heuristics to support our definition of tangent space.

\subsection{Heuristics}\label{sec:heuristics}
\subsubsection{What can be found in the literature}
A similar notion of tangent space was in fact considered in \citep{giglitesi} and Section 12.4 in \citep{ags} (the so-called velocity plans), with the restriction to maps $\psi$ satisfying some sort of optimality, arising from the fact that only couplings arising from geodesics of optimal transport were considered. In regard to their definition, our main contribution is simply to avoid this restriction of $T_µ(\mpt)$. Furthermore, we point out the work \citep{giglitesi,aussedat} which, despite not identifying the same tangent space as us, studies it and provides interesting results concerning the decomposition of $T_µ(\mpt)$ into \eqref{bad} and what could be interpreted as its orthogonal space.\\

Somehow, the comparison result we established in \citep{bertucci2024stochastic} is a natural application of the choice of tangent space. Indeed, by considering extended super-differentials (see Section 10.3 in \citep{ags}), we showed that comparison principles are immediate for Hamilton-Jacobi equation in $\mathcal P (\mathbb{T}^d)$. Hence, we allowed to seek for elements in super-differentials which are in $T\mpp$, and not only in some set defined as in \eqref{bad}. This sort of ideas were later also used in \citep{ceccherini1,ceccherini2}.\\

More generally, the notion of extended super-differentials of \citep{ags} basically seeks to define what does it mean for an element of $T_µ(\mpp)$ to be in the super-differential of a function $U:\mpp \to \R$ at the point $µ$.\\

Furthermore, recent developments \citep{cavagnari1,cavagnari2,cavagnari3} in the precise characterization of not so smooth dynamics on $\mpp$ have made a systematic use of so-called multi valued probability vector fields, which are somehow vector fields on the tangent space we defined above. 

\subsubsection{Young measures}
The main issue when dealing with \eqref{bad} is that the latter space behaves quite badly with $µ$. The next example is a good illustration of this phenomenon.
\begin{Ex}\label{ex:1}
Let $d = 1$, $µ = \delta_0$ and consider for $\eps > 0$, $µ_\eps:= \mathcal N(0,\eps)$, the Gaussian distribution of mean $0$ and variance $\eps$. Clearly $W_2(µ,µ_\eps) \to 0$ as $\eps \to 0$. For any $\eps > 0$, the function $f (x) = \text{sign}(x) \text{ if } x \ne 0$, is in $ L^2((\R,µ_\eps),\R)$. However, the statement $f \in L^2((\R,µ),\R)$ is far from being clear. More generally $L^2((\R,µ),\R)$ has not much to do with $L^2((\R,µ_\eps),\R)$, even when $\eps$ is small but non-zero.
\end{Ex}
Hence, it seems that when dealing with \eqref{bad}, we loose a lot of information at some point compared to what can happen in a close neighborhood. In our opinion, it is reminiscent of an old problem faced in calculus of variations which has been solved by Young \citep{young}. He introduced what are now called Young measures, which allow to capture some oscillatory, random or complex behavior of sequences of functions by considering objects similar to the ones of \eqref{tangent}. Usually in Young measures, the reference measure $µ$ is fixed and we are simply taking limits of functions. Here, we claim that if we allow $µ$ to vary also, then we naturally end up with $T_µ(\mpp)$ if we want to keep all this information.

Typical problems in Young measures are to characterize the limit of functions $f_n: [0,1]\to \R$ defines as the almost everywhere derivative of $g_n$, a piecewise affine function which takes value $0$ on $\frac{k}{2^n}$ if $k$ is even, and value $1$ if $k$ is odd, and which is smooth outside of these points. Typically, Young measures allow to state that, in some sense, $f_n$ converges toward the measured valued map $\psi:[0,1] \to \mathcal P(\R)$ which is constant and takes value $\frac12(\delta_{-1} + \delta_1)$. 

In Example \ref{ex:1}, in some sense, it is natural to say that the limit version of $f \in L^2((\R,µ_\eps),\R)$ is to say that the function defined only on $\{0\}$ by $\big(0 \to \frac12(\delta_{-1} + \delta_1)\big)$ is an element of $ T_µ(\mathcal P(\R))$.

The point of view of Young measures can be interpreted as the fact that a relaxation of a functional space usually leads to considering measured value functions. This is quite in the spirit of the powerful Kantorovich relaxation of optimal transport, in fact it is an omnipresent idea in optimization, analysis or calculus of variations that we feel deserved to be highlighted here.

\subsubsection{Sheaf structure of the tangent spaces}\label{sec:sheaf}
When trying to equip $(\mpt,W_2)$ with the tangent space of \eqref{bad}, it is clear that we are quite far from obtaining the structure of a tangent bundle. Indeed, the nature of the candidate tangent space varies too much with $µ$ and makes any sort of local trivialization impossible.

However, the identification of $T_µ(\mpt)$ with a subset of $\mathcal P(\R^{2d})$ that we proposed hints a sheaf structure for this tangent space, that we establish in the next proposition.
\begin{Prop}
The functor $\mathcal F$ which associates to any open set $U \subset (\mpt,W_2)$ $\mathcal F(U) := \{\varphi: U \to \mathcal P_2(\R^{2d}) | \varphi \text{ is continuous }, (\pi_1)_\# \circ \varphi = Id\}$ defines a sheaf on $\mpt$ when associated to the standard restriction morphisms. Furthermore, the stalk $\mathcal F_µ$ of $\mathcal F$ over some $µ \in \mpt$ is given by $\{\gamma \in \mathcal P_2(\R^{2d}) | (\pi_1)_\# \gamma = µ\}$, that is by $T_µ(\mpt)$.
\end{Prop}
\begin{proof}
The result is trivial as all the properties of a sheaf are automatically verified. Furthermore, the map $(\pi_1)_\# \cdot$ is continuous on $\mathcal P_2(\R^{2d})$, which naturally yields the structure of the stalk.
\end{proof}

\subsubsection{Lagrangian vs Eulerian dynamics}
In finite dimension, the tangent space to a manifold $M$ at a point $x \in M$ can be thought of as being the set of speeds $\dot \theta(0)$ of smooth curves $\theta$ such that $\theta(0) = x$. In $\mpp$, restricting our attention to geodesics and having in mind the Benamou-Brenier dynamic reformulation of optimal transport \citep{benamou2000computational}, we know that, formally geodesics $(m_t)_{t \in [-1,1]}$ are given as solution of continuity equations
$$
\partial_t m + \text{div}_x(\nabla_x \phi(t,x)m) = 0 \text{ in } (-1,1)\times \R^d,
$$ 
for some real valued function $\phi: [-1,1]\times \R^d \to \R$. Hence, the space \eqref{bad}, which thus inherits the Eulerian formulation of geodesics on $\mpp$, is the closure of possible speeds $\nabla_x \phi(0,\cdot)$.\\

A more Lagrangian approach consists in looking at dynamics $(m_t)_{t \in [-1,1]}$ expressed as the law of processes $(X_t)_{t \in [-1,1]}$, on a standard probabilistic space $(\Omega, \mathcal A,\mathbb P)$. At a Lagrangian level, smoothness implies that $X_t(\omega)$ is a smooth function of $t$, for almost every $\omega$. In this context, the speed of the curve is the collection $(\dot X_t(\omega))_{\omega \in \Omega}$. Keeping track of all the speeds then amounts to being interested in the set of pairs $(X_0(\omega),\dot X_0(\omega))_{\omega \in \Omega}$, and the fact that we are only interested in the law of such objects leads us to consider $\gamma := \mathcal L(X_0,\dot X_0) \in \mathcal P(\R^{2d})$. We expand on this in Sections \ref{sec:lift} and \ref{sec:curves} below.

\subsubsection{Particles approximation}
In order to tackle problems on $\mpp$, approximations of measures by combinations of $N$ Dirac masses can be considered before passing to the limit $N \to \infty$. At the level $N < \infty$, we want to approximate $\mpp$ with $(\R^d)^N$ equipped with an appropriate $\ell^p$ topology. For $x \in (\R^d)^N$, the tangent space $T_x((\R^d)^N)$ is simply given by $(\R^d)^N$. Let $(x_N)_{N \geq 1}$ be a sequence such that $x_N \in (\R^d)^N$ and $\lim_{N \to \infty} \frac1N \delta_{x_N^i} = µ \in \mpp$ and $(y_N)_{N\geq 1}$ be such that $y_N \in T_{x_N}(\R^d)^N$ for all $N\geq 1$. We then say that $(y_N)$ is going to approximate an element $\psi \in T_µ(\mpp)$ if $\lim_{N \to \infty} \frac1N \delta_{(x_N^i,y_N^i)} = µ(dx)\psi_x(dz)$. We then observe the difference with eventual deterministic tangent spaces such as \eqref{bad} as the fact that there is no restriction of the type $x_N^i= x_N^j\Rightarrow y_N^i = y_N^j$.

\section{Parallel transport on $\mpt$}\label{sec:para}
The first application of our concept of tangent space is to define a notion of parallel transport on $\mpp$. The notion of parallel transport usually requires three main concepts such as: tangent spaces (or more generally vector fields), a notion of smooth curves, and a notion of derivative along a smooth curve $\theta$. A parallel transport is then defined as a curve valued in the tangent space over the curve $\theta$ whose derivative in the direction that follows $\theta$ is $0$. We make precise such an idea below.

We shall first present the notion of smooth curves along which are going to parallel transport our vectors. As the geometry of $\R^d$ is euclidean, the notion of derivative along a curve shall be quite trivial. The case of $\mathcal P(M)$ treated below will require some care in this aspect. Furthermore, we restrict the following to the case $p =2$ to lighten notation and comment at the end of this section on the case $p \ne 2$.

We finish this section by comparing our notion to the one introduced and studied in \citep{gigli}.

\subsection{Smooth curves}
A curve on $\mpt$ is typically a map $\theta : [0,1] \to \mpt$. Saying that a curve is continuous or Lipschitz is clear in this setting. We want to consider curves which are more regular than simply Lipschitz. Since $\mpt$ is not a proper normed vector space, there is no intrinsic way to do so and we explain below the underlying difficulties and the choice we make.

\subsubsection{What we do not do}
Given the nature of $\mpt$, it is not obvious how to define $C^1$ curves. For instance, we could understand curves $\theta$ satisfying $\theta'(t) = cst$, when embedding $\mpt$ in the space of signed measures equipped with total variation. But by doing so we would miss curves of the form $(\delta_{x(t)})_{t \in [0,1]}$ for $(x(t))_{t \in [0,1]}$, a smooth curve on $\R^d$. The latter type of curve is typically characterized by being the unique solution of equations of the form
\be\label{ce}
\partial_t m + \text{div}_x(\phi(t,x) m) = 0 \text{ on } (0,1)\times \R^d,
\ee
where $\text{div}_x$ is the divergence operator on $\R^d$ and $\phi: [0,1]\times \R^d \to \R^d$ is a smooth vector field, Lipschitz in $x$, say uniformly in $t$. However, restricting our attention to this type of curves would make us miss the geodesics of the distance $W_2$, which are typically characterized as being the solution of such equations for non-smooth functions $\phi$. The latter makes the discussion about what is a smooth curve quite complex, thus we do not adopt this point of view here. 

\subsubsection{What we do}
We take a so-called Lagrangian point of view and consider more general curves $\theta$ on $\mpt$, namely one which are of the form $\theta(t) = (e_t)_\#\eta$ for some $\eta \in \mathcal P(\mathcal C([0,1],\R^d))$, where $e_t(f) = f(t)$, for $f \in \mathcal C([0,1],\R^d)$, i.e. $e_t$ is the evaluation mapping at time $t$. We say that $\eta$ is $\mathcal C^1$ if it is concentrated on $\mathcal C ^1$ curves. Such a notion of regularity does not naturally translate to a regularity for $\theta$, for instance because of the clear multiplicity of $\eta$ which generate the same $\theta$.

In fact, through Ambrosio's superposition principle which is presented in details in Section \ref{sec:trans}, there is a correspondance between this Lagrangian point of view and the previous Eulerian one. It just happens that the Lagrangian point of view is much more convenient for what follows.

\subsection{Vector fields and their derivatives}
A vector field above $\mpt$ is typically a map $V$ defined on $\mpt$ such that $V(µ) \in T_µ(\mpt)$ for all $µ$. To define parallel transport, we are interested in vector field along a curve, which we decided to be given by a Lagrangian formulation, hence we shall naturally be interested in Lagrangian vector field. 

A vector field along a Lagrangian curve $\eta \in \mathcal P(\mathcal C([0,1],\R^d))$ is a probability measure $\Psi \in \mathcal P(\mathcal C[0,1],\R^{2d})$ such that $(p_1)_\#\Psi = \eta$, where $p_1: \mathcal C([0,1],\R^{2d}) \to  \mathcal C([0,1],\R^{d})$ is defined by $p_1(z)(t) = (\pi_1)(z(t))$. If $\Psi$ is concentrated along smooth curves in $\R^{2d}$, its derivative along $(p_1)_\#\Psi$ is by definition the curve $(\partial_2)_\#\Psi$ where $\partial_2: \mathcal C^1([0,1],\R^{2d}) \to  \mathcal C([0,1],\R^{d})$ is defined by $\partial_2(z)(t) = \frac{d}{dt}(\pi_2(z(t)))$. More generally, $\partial_2$ is well defined on curves on $\R^{2d}$ such that the second component is in $\mathcal C^1$. Such a space is denoted $\mathcal C_y^1([0,1],\R^{2d})$.

\subsection{Main definition and properties}
Since the geometry of $\R^d$ is rather trivial, the parallel transport in $\R^d$ simply consists in translating the tangent vector. At the level of $\mpt$, this simplicity translates into the following definition.

\begin{Def}
Let $µ \in \mpt$, $\psi \in T_µ(\mpt)$ and $\eta \in \mathcal P (\mathcal C([0,1],\R^d))$ such that $(e_0)_\#\eta = µ$. A parallel transport of $\psi$ along $\eta$ is a measure $\Psi \in \mathcal P(\mathcal C([0,1],\R^{2d}))$ concentrated on $\mathcal C^1_y([0,1],\R^{2d})$ such that 
\begin{itemize}
\item $(p_1)_\#\Psi = \eta$,
\item $(\partial_2)_\#\Psi = \delta_0$,
\item $(e_0)_\#\Psi(dx,dz) = µ(dx)\psi_x(dz)$.
\end{itemize}
\end{Def}
\begin{Rem}
Remark that no integrability condition is required here, as it follows from the definition as the next result shows.
\end{Rem}
The following result makes transparent that a parallel transport yields a curve in the tangent sheaf $T\mpt$.
\begin{Prop}\label{prop:inte}
Let $µ \in \mpt$, $\psi \in T_µ(\mpt)$, $\eta \in \mathcal P (\mathcal C([0,1],\R^d))$ such that $(e_0)_\#\eta = µ$ and $\Psi$ a parallel transport of $\psi$ along $\eta$. Then, for all $t \in [0,1]$, $(e_t)_\#\Psi = (e_t)_\#\eta(dx)\psi_{t,x}(dz)$ for some $\psi_t \in T_{(e_t)_\#\eta}(\mpt)$.
\end{Prop}
\begin{proof}
Let $(X,Z)$ be a process of law $\Psi$. For $t \in [0,1]$, $(X_t,Z_t)$ is such that $\mathcal L(X_t) = (e_t)_\#\eta$. Moreover, since $Z_t = Z_0$ and that $\mathcal L(Z_0) \in \mpt$ since $\psi \in T_µ(\mpt)$, the result follows.
\end{proof}

The existence of parallel transport is proved next to be a simple consequence of the gluing Lemma.
\begin{Prop}\label{prop:exist}
Let $µ \in \mpt$, $\psi \in T_µ(\mpt)$ and $\eta \in \mathcal P (\mathcal C([0,1],\R^d))$ such that $(e_0)_\#\eta = µ$. There exists a parallel transport of $\psi$ along $\eta$.
\end{Prop}
\begin{proof}
By gluing together $(e_0,Id)_\#\eta(d\theta)$ and $µ(dx)\psi_x(dz)$ along their first coordinate, we obtain a probability measure $\Gamma(dx,d\theta,dz)$. Consider the map $\mathcal T: \R^d \times \mathcal C([0,1],\R^d) \times \R^d\to\mathcal C([0,1],\R^{2d}), (x,\theta,z) \to (t \to (\theta(t),z))$. The probability $\mathcal T_\#\Gamma$ is a required parallel transport. 
\end{proof}
The next example shows what parallel transport is along curves which are simply constructed by interpolation of a coupling between two measures. In particular, such curves model all geodesics in $\mpt$.
\begin{Ex}
Let $µ \in \mpt, \psi \in T_µ(\mpt)$. Let $\eta \in \mathcal P (\mathcal C([0,1],\R^d))$ being defined as the law of $((1-t)X' + t Y')$ for $(X',Y')$ a couple of random variables of joint law $\gamma \in \mathcal P_2(\R^{2d})$. Then the construction of the parallel transport is somehow more understandable. Consider $(X,Y,Z)$ random variables such that $\mathcal L(X,Y) = \gamma$ and $\mathcal L(X,Z) = µ(dx)\psi_x(dz)$. Then a parallel transport of $\psi$ along $\eta$ is given by the law of $t \mapsto ((1-t)X + tY,Z)$. Note that here also, the existence of $(X,Y,Z)$ is guaranteed by the gluing Lemma.
\end{Ex}
The next result relies heavily on the fact that the parallel transport in $\R^d$ is trivial, in the sense that it only depends on the starting and end points.
\begin{Prop}\label{prop:analogchemincoupl}
Let $µ \in \mpt$, $\psi \in T_µ(\mpt)$, $\eta_1,\eta_2 \in \mathcal P (\mathcal C([0,1],\R^d))$ such that $(e_0)_\#\eta_i = µ$ for $i = 1,2$ and $\Psi$ a parallel transport of $\psi$ along $\eta_1$. If $(e_0,e_1)_\#\eta_1 = (e_0,e_1)_\#\eta_2$, then $\Psi$ is also a parallel transport of $\psi$ along $\eta_2$.
\end{Prop}
We leave this immediate proof to the reader. 
\begin{Rem}
This result justifies that, when working on probability measures over $\R^d$, we can talk about a parallel transport $\Psi$ of $\psi \in T_µ(\mpt)$ along a coupling $\gamma \in \mathcal P(\R^{2d})$ when $(\pi_1)_\#\gamma = µ$.
\end{Rem}

In general, there is no uniqueness of a parallel transport with our definition as the next example shows.

\begin{Ex}\label{ex:nonuniq}
Let $d=1$ and $µ := \delta_0$. Let $\eta$ be given by $\frac12(\delta_{\theta^1} + \delta_{\theta^2})$ where $\theta^i(t) = 2(\frac32-i)t$. Let $\psi_0 = \frac{1}{2}(\delta_0 + \delta_2)$. Then the laws of the three following processes are all parallel transports of $\psi$ along $\eta$,
$$
\begin{cases}
\{W(t) = (t,0), t \in [0,1]\} \text{ with probability } \frac12,\\
\{W(t) = (-t,2), t \in [0,1]\} \text{ with probability } \frac 12,
\end{cases}
\begin{cases}
\{M(t) = (t,2), t \in [0,1]\} \text{ with probability } \frac12,\\
\{M(t) = (-t,0), t \in [0,1]\} \text{ with probability } \frac 12,
\end{cases}
$$
$$
\begin{cases}
\{L(t) = (t,2), t \in [0,1]\} \text{ with probability } \frac14,\\
\{L(t) = (t,0), t \in [0,1]\} \text{ with probability } \frac14,\\
\{L(t) = (-t,0), t \in [0,1]\} \text{ with probability } \frac 14,\\
\{L(t) = (-t,2), t \in [0,1]\} \text{ with probability } \frac 14.
\end{cases}
$$
\end{Ex}
This example shows that as soon as there are multiple ways to obtain our main coupling through the gluing Lemma in the proof of Proposition \ref{prop:exist}, there will be multiple parallel transports. On the converse, the next result shows that when one of the two couplings involved in the proof of Proposition \ref{prop:exist} is deterministic, then there is a unique parallel transport. Similar remarks about the uniqueness of the law given by the gluing Lemma are of course already well-known, see for instance Lemma 5.3.2 in \citep{ags}.
\begin{Prop}\label{prop:uniqpara}
Let $µ \in \mpt$, $\psi \in T_µ(\mpt)$ and $\eta \in \mathcal P (\mathcal C([0,1],\R^d))$ such that $(e_0)_\#\eta = µ$. Assume that at least one of the two following statements holds.
\begin{itemize}
\item The function $\psi$ is of the form $x \to \delta_{v(x)}$ for some measurable $v \in L^2_µ(\R^d,\R^d)$.
\item There exists $f \in L^\infty_t(W^{1,\infty}(\R^d,\R^d))$ such that $\eta$ is concentrated on curves $\theta$ solutions of $\theta'(t) = f(t,\theta(t))$.
\end{itemize}
Then, there exists a unique parallel transport of $\psi$ along $\eta$.
\end{Prop}
\begin{proof}
Consider two parallel transport $\Psi_1$ and $\Psi_2$ of $\psi$ along $\eta$, and two processes $(X^1,Z^1), (X^2,Z^2)$ of respective law $\Psi_1$ and $\Psi_2$.

In the first case, we find that for $i=1,2$, the process $(X^i,Z^i)$ is of the form $(X^i_t,v(X^i_0))_{t \in [0,1]}$. Since $X^1$ and $X^2$ both have law $\eta$, it follows that $\Psi_1= \Psi_2$.

In the second case, since the flow $\phi: [0,1]\times \R^d \to \R^d$ associated to the ODE $\theta'(t) = f(t,\theta(t))$ is well defined, it follows that $X^i$ is of the form $(\phi(t,X^i_0))_{t \in [0,1]}$. Hence, for $i=1,2$, the process $(X^i,Z^i)$ is of the form $((\phi(t,X^i_0),Z^i))_{t \in [0,1]}$. Since $(X^1_0,Z^1)$ and $(X^2_0,Z^2)$ have law $µ(dx)\psi_x(dz)$, it also follows that $\Psi_1=\Psi_2$.
\end{proof}
\begin{Rem}
More general condition on $\eta$ are sufficient. In fact, it suffices in this case, that $\eta$ is concentrated on curves $\theta$ such that there exists a measurable map $T: \R^d\to \R^d$ such that $\theta(1)= T(\theta(0))$.
\end{Rem}

Since parallel transport is uniquely defined when we start from a deterministic element of the tangent space, one might wonder why bother with the "extended" tangent space that we proposed. The next example shows that, even if we start with a deterministic element of the tangent space, and thus that the parallel transport is uniquely defined, the transported element might end up being a non-deterministic element of the tangent space at the final point.

\begin{Ex}
Let $d=1$, $µ$ be the normal distribution and $\psi: x \to \delta_{x} \in T_µ(\mathcal P_2(\R))$. Let $\eta\in \mathcal P(\mathcal C([0,1],\R))$ be given by the coupling $µ(dx)\delta_0(dy)$. Then, the (unique) parallel transport $\Psi$ of $\psi$ along $\eta$ is such that $(e_1)_\#\Psi(dx,dz)= \delta_0(dx)µ(dz)$. In particular, the associated element $\psi' \in T_{\delta_0}(\mathcal P_2(\R))$ is not deterministic, as it associates $µ$ to the point $0$.
\end{Ex}

We also have the following obvious time reversibility of parallel transports.
\begin{Prop}
Let $\tau(\theta)$ be the reverse parametrization by time of a curve $\theta \in \mathcal C([0,1],X)$ for some topological space $X$, i.e. $\tau(\theta)(t) = \theta(1-t)$.

If $\Psi$ is a parallel transport along $\eta$, then $\tau_\#\Psi$ is a parallel transport along $\tau_\#\eta$.
\end{Prop}

Finally, we present another example to highlight the fact that our notion of parallel transport indeed depends on the Lagrangian path.

\begin{Ex}
Let $µ$ be the uniform distribution on $[0,1]$, and consider $\psi : x \to \delta_x \in T_µ(\mpt)$. Let $U$ be a random variable of law $µ$, $\eta_1 := \mathcal L( (U)_{t \in [0,1]})$ and $\eta_2= \mathcal L((t U + (1-t)(1-U))_{t \in [0,1]})$. Then, for all $t \in [0,1], (e_t)_\#\eta_1 = (e_t)_\#\eta_2$, however, the parallel transport of $\psi$ along $\eta_1$ and $\eta_2$ are clearly different.
\end{Ex}

\subsection{The case of $\mpp$}
The extension of the above to $\mpp$ does not raise any particular difficulty, mainly because the integrability conditions follows from the definition, as Proposition \ref{prop:inte} shows. Hence, we do not give more details here.

\subsection{Link with a previous notion of parallel transport on $\mpt$}
In \citep{gigli,gigli2}, the authors introduced another notion of parallel transport on $\mpt$ (and in fact on the space of probability measure over a manifold as well) that we do not present in full details, but rather give our interpretation of it, in order to compare it with our notion.

In \citep{gigli,gigli2}, the authors are mostly interested with what we can call Eulerian paths on $\mpt$, that is those that can be represented as solution to a continuity equation of the form
$$
\partial_t m + \text{div}_x(\alpha_tm) = 0 \text{ in } (0,1)\times \R^d,
$$
where $(m_t)_{t \in [0,1]} \in \mathcal C([0,1],\mpt)$ is the path in question and $\alpha : [0,1]\times \R^d \to \R^d$, the velocity that moves the underlying particles. We refer to Section \ref{sec:curves} below for more insights and references on such dynamics. 

The authors then show that under certain regularity assumptions on $\alpha$, the ODE $\dot x = \alpha(t,x)$ defines a sufficiently smooth flow, along which they can then parallel transport the deterministic elements of the tangent space, while projecting it onto the space of gradients, namely to ensure that the new vector field will have a $0$ rotational.

Our point of view is that the two approaches coincide (up to the projection that we do want to use here) when the one of \citep{gigli,gigli2} is well defined, but that we are able to avoid quite a lot of technical details because we do not use the Eulerian formulation of the path at interest. Of course, the downside of our approach is that we lose uniqueness of parallel transport in certain pathological cases. Furthermore, we believe our different examples justify this loss of uniqueness as natural.

\section{Links with Lions' lift}\label{sec:lift}
The definitions and results of the previous section might seem rather obscure in terms of their notation. This section somehow enlightens them by making transparent the translation in terms of random variables, following the powerful lifting approach developed by P.-L. Lions in \citep{lions}, that we now recall. Let $(\Omega,\mathcal A, \mathbb P)$ be a standard atomeless probabilistic space and  $H = L^2(\Omega,\R^d)$ the separable Hilbert space of square integrable $\R^d$ valued random variables on $\Omega$. The point of view developed by P.-L. Lions, mainly to study partial differential equations over $\mpt$, is simply to lift problems on $\mpt$ into problems on $H$, an approach that we follow in this section.

Since $H$ is flat, parallel transport in $H$ is rather obvious. The tangent plan at $X \in H$ in always $H$ itself and the parallel transport of a vector $Z \in H$ above some initial point $X \in H$ along a curve $\theta: [0,1]\to H, t \to \theta(t):=X_t$ is the constant function $Z$.

The main link between the two approaches is the following proposition, which is a tautology.
\begin{Prop}
\begin{itemize}
\item Consider a $\R^d$ valued process $(X_t)_{t \in [0,1]}$, which is almost surely continuous in $t$, such that $\sup_{t \in [0,1]}\|X_t\|_H < \infty$ and $Z \in H$. Note $\mathcal L(X_0,Z):= µ(dx)\psi_x(dz)$. Then, $\mathcal L((X_t,Z)_{t \in [0,1]})$ is a parallel transport of $\psi$ along $\mathcal L((X_t)_{t \in [0,1]})$.
\item Conversely, consider $µ \in \mpt$, $\psi \in T_µ(\mpt)$ and $\eta \in \mathcal P(\mathcal C([0,1],\R^d))$ such that $(e_0)_\#\eta = µ$ and a parallel transport $\Psi$ of $\psi$ along $\eta$. Then, any $(X_t,Z)_{t \in [0,1]}$ of law $\Psi$ is the parallel transport in $H$ of $Z$ along $(X_t)_{t \in [0,1]}$.
\end{itemize}
\end{Prop}

The above transposition is so simple that it can be tempting to restrict ourselves to working on $H$ and then pull back all the information we have at the level of $\mpt$. This might be misleading if we work with probability measures on other domains than $\R^d$, as the underlying Hilbertian structure collapses. The next section extends this notion of parallel transport to cases of probability measures on smooth domain of $\R^d$ or on manifolds, showing that in Lions' approach, the lift is more important than the Hilbertian structure. 

\section{Tangent space and parallel transport on $\mathcal P(\mo)$}
We expand rapidly the notion of parallel transport on $\mpo$ when $\mo$ is either a smooth bounded domain of $\R^d$ or a smooth compact Riemannian manifold $M$.

\subsection{The case of a smooth domain of $\R^d$}
Assume that $\mo$ is the closure of a smooth bounded domain of $\R^d$. Because $\mo$ is bounded, we simply work on $\mpo$ and there is no need to specify integrability conditions on the measures. In this somehow simpler case, the geometry of $\mo$ is still flat, hence the parallel transport on $\mo$ is trivial. 

In this case, we define the tangent space of $µ \in \mpo$ as
$$
T_µ(\mpo) := \left\{ \psi: \mo \to \mathcal{P}(\R^d) \bigg| \int_{\bar \mo \times \R^d}|z|\psi_x(dz)µ(dx) < \infty \right\}.
$$

Note that the integrability condition is of the $L^1$ type, this is natural from a duality standpoint, since $\mo$ is bounded. This notion of tangent space can be justified by the same heuristics as the ones we gave in Section \ref{sec:heuristics}. Moreover, in this case, it can be identified with $\{\gamma \in \mathcal P_1(\mo\times \R^d)| (\pi_1)_\#\gamma = µ\}$.

Parallel transport is here defined as follows.
\begin{Def}
Let $µ \in \mpo$, $\psi \in T_µ(\mpo)$ and $\eta \in \mathcal P (\mathcal C([0,1], \mo))$ such that $(e_0)_\#\eta = µ$. A parallel transport of $\psi$ along $\eta$ is a measure $\Psi \in \mathcal P(\mathcal C([0,1], \mo\times\R^{d}))$ concentrated on $C^1_y([0,1], \mo \times \R^{d})$ such that 
\begin{itemize}
\item $(p_1)_\#\Psi = \eta$.
\item $(\partial_2)_\#\Psi = \delta_0$.
\item $(e_0)_\#\Psi(dx,dz) = µ(dx)\psi_x(dz)$.
\end{itemize}
\end{Def}

All the results and example established in the case of $\mpt$ follow in a similar manner and their details are omitted to avoid repetitions.

\subsection{The case of smooth Riemannian manifold}
Let $(M,g)$ be smooth bounded Riemannian manifold. The tangent space to $M$ at $x\in M$ is denoted $T_xM$, and $TM$ is the tangent bundle. The notion of parallel transport on $(M,g)$ is the one associated to the Levi-Civita connection. Recall that given a $C^1$ curve $\theta:[0,1]\to M$, and $\xi \in T_{\theta(0)}M$, there exists a unique parallel transport of $\xi$ along $\theta$, and it will be noted $\tau_\theta(\xi)$, which is, to fix a convention here, a curve in $TM$.

For $µ \in \mm$, the tangent space to $\mm$ at $µ$ is 
\be\label{defM}
\ba
T_µ(\mm) := \bigg\{&\psi: M \to \cup_{x \in M}\mathcal P(T_xM), \text{ measurable, } \\
&\bigg| \forall x \in M, \psi_x \in \mathcal P(T_xM), \int_M \int_{T_xM}|z|_x\psi_x(dz)µ(dx) < \infty\bigg\}.
\ea
\ee
The interpretation of the previous is quite clear: we keep the idea of allowing for a probabilistic description of the tangent vectors over every points. The integrability condition is in $L^1$ because $M$ is bounded. In this setting, one can identify, for $µ\in \mm$, $T_µ(\mm)$ with $\{\gamma \in \mathcal P_1(TM)| (\pi_1)_\#\gamma = µ\}$. Note that the norm $|z|_x$ in \eqref{defM} is the one of $T_xM$ induced by $g_x$. 

In this setting, parallel transport on $\mm$ is defined as follows.
\begin{Def}
Let $µ \in \mm$, $\psi \in T_µ(\mm)$ and $\eta \in \mathcal P (\mathcal C([0,1],M))$, concentrated on $C^1$ curves and such that $(e_0)_\#\eta = µ$. A parallel transport of $\psi$ along $\eta$ is a measure $\Psi \in \mathcal P(\mathcal C([0,1],TM))$ concentrated on $C^1$ curves such that 
\begin{itemize}
\item $(p_1)_\#\Psi = \eta$.
\item $\Psi $ is concentrated on parallel transports on $M$.
\item $(e_0)_\#\Psi(dx,dz) = µ(dx)\psi_x(dz)$.
\end{itemize}
\end{Def}
Once again, similar results as in the $\mpt$ case hold and their proof are similar so we do not give them, except for the following integrability and existence results.
\begin{Prop}
Let $µ \in \mm$, $\psi \in T_µ(\mm)$ and $\eta \in \mathcal P (\mathcal C([0,1],M))$, concentrated on $C^1$ curves, such that $(e_0)_\#\eta = µ$ and $\Psi$ a parallel transport of $\psi$ along $\eta$. Then, for any $t \in [0,1]$, $(e_t)_\#\Psi \in \mathcal P_1(TM)$. 
\end{Prop}
\begin{proof}
Since parallel transport on $M$ preserves the norm, if we consider a process $(X_t,Z_t)_{t \in [0,1]}$ of law $\Psi$, we obtain that almost surely $t \mapsto |Z_t|_{X_t}$ is constant in $[0,1]$. Hence, the result follows by taking the expectation of this quantity.
\end{proof}
\begin{Prop}
Let $µ \in \mm$, $\psi \in T_µ(\mm)$ and $\eta \in \mathcal P (\mathcal C([0,1],M))$, concentrated on $C^1$ curves and such that $(e_0)_\#\eta = µ$. There exists a parallel transport of $\psi$ along $\eta$.
\end{Prop}
\begin{proof}
Once again, existence relies on the gluing Lemma, which we use to get the existence of a probability measure $\Gamma(dx,dp,d\theta) \in \mathcal P(M\times TM\times \mathcal C([0,1],M))$ such that $(\pi_3)_\#\Gamma = \eta$, $(\pi_1,\pi_2)_\#\Gamma = µ(dx)\psi_x(dz)$ and $(\pi_1,e_0\circ\pi_3)_\#\Gamma = (Id,Id)_\#µ$. Then, $(\tau_{\pi_3}(\pi_2))_\#\Gamma$ is a required parallel transport.
\end{proof}

\begin{Rem}
In order to use the notion of parallel transport on $M$, we need to restrict ourselves to $\eta \in \mathcal P (\mathcal C([0,1],M))$ being supported on $C^1$ curves. Recall that parallel transport on $M$ cannot be defined along a curve which is simply continuous. Note finally that this regularity shall not guarantee the uniqueness of parallel transport on $\mm$, see Example \ref{ex:nonuniq} which can be adapted to this case for instance.
\end{Rem}

\section{Comparing elements in different tangent spaces and regularity of elements of $\mpo \to \R$}\label{sec:c1a}
One of the main advantages of parallel transport is that it allows to compare elements in different tangent spaces, namely by parallel transporting one into the tangent space of the other, usually along a geodesic. We explain here how such an idea can be used to define regularity classes for functions $U: \mathcal P \to \R$. In particular, we are interested in $\mathcal C^{1,\alpha}$ spaces, that is in the regularity of the differential, which is naturally an element of the tangent space. For the sake of clarity, we restrict ourselves to the case of $\mpt$ here, and discuss extensions at the end of the section. 

We organise our results as follows. We start by introducing a natural notion of distance on the tangent space $T_µ(\mpt)$ for some fixed $µ \in \mpt$ and show how, with the help of parallel transport, we can construct a way of comparing elements in different tangent spaces. We then proceed by giving another natural way of comparing elements in different tangent spaces, by equipping $T\mpt:= \cup_{µ \in \mpt} \{µ\}\times T_µ(\mpt)$ with a distance directly, commenting on the links with the first approach. We then proceed to define regularity spaces for functions $\mpt \to \R$.

The first steps of our program are clearly aligned with basic objects in differential geometry such as metric on both the base space and the tangent bundle (or sheaf here).

\subsection{A distance on the tangent space}
Before comparing two elements in different tangent spaces, we need to be able to do so for elements in the same one. Let $µ \in \mpt$. Recall that there is an equivalence between choosing $\psi \in T_µ(\mpt)$ and choosing $\gamma \in \mathcal{P}(\R^{2d})$ such that $(\pi_1)_\#\gamma = µ$, where $\gamma(dx,dz) = µ(dx)\psi_x(dz)$. This could argue in favor of using a simple Wasserstein distance between $µ(dx)\psi_x(dz)$ and $µ(dx)\psi'_x(dz)$, for $\psi,\psi' \in T_µ(\mpt)$, however this does not work and we refer to the proof of Proposition \ref{prop:comppara} below for a convincing argument. The distance we shall use requires a bit more rigidity. It was already considered by Gigli and it is given by either a probabilistic reformulation or a coupling one.\\

For $\psi,\psi'\in T_µ(\mpt)$, define $d_µ(\psi,\psi')$ by
$$
d^2_µ(\psi,\psi') := \inf_{\gamma}\int_{\R^d\times \R^d\times \R^d}|z-z'|^2\gamma(dx,dz,dz'),
$$
where the infimum is taken over probability measures $\gamma \in \mathcal P(\R^d\times \R^d\times \R^d)$ such that $(\pi_1,\pi_2)_\#\gamma = µ(dx)\psi_x(dz)$ and $(\pi_1,\pi_3)_\#\gamma = µ(dx)\psi'_x(dz)$. Equivalently, this quantity can be defined by 
$$
d^2_µ(\psi,\psi') := \inf_{(X,Z,Z')}\mathbb{E}[|Z-Z'|^2],
$$
where the infimum is taken over triplets of random variables $(X,Z,Z')$ such that $\mathcal L(X,Z)= µ(dx)\psi_x(dz)$ and $\mathcal L(X,Z')= µ(dx)\psi'_x(dz)$.

Note that if $\psi$ and $\psi'$ are deterministic, represented by respectively $\alpha,\alpha' \in L^2_µ(\R^d,\R^d)$, then 
$$
d^2_µ(\psi,\psi') = \|\alpha-\alpha'\|^2_{L^2_µ}.
$$

This distance has already been been introduced in \citep{giglitesi}, and it can also defined a sort of metric on $T_µ\mpt$ through the relation $\langle \psi,\psi'\rangle_µ := d^2_µ(\psi,\delta_0) + d^2_µ(\delta_0,\psi') - d^2_µ(\psi,\psi')$.

\subsection{Comparing elements in different tangent spaces}
Building on this distance, there is a natural way to compare two elements in different tangent spaces $T_µ\mpt$ and $T_\nu\mpt$: parallel transport one along a path connecting $µ$ and $\nu$ and use the previous distance. Note that there is a choice to make, as there are several paths connecting $µ$ and $\nu$ and there are no reason that they should give the same result. Furthermore, there is no uniqueness of the parallel transport either, even if the path is fixed, as we saw earlier. Quite naturally, we focus on geodesics here, but we shall see later that looking at more paths can also be natural. In any case, we are interested in the following quantity.

\begin{Def/Prop}
We note $\mathcal E: T\mpt \times T\mpt \to \R_+$ the map defined by either
\be\label{defE}
\mathcal E((µ,\phi),(\nu,\psi)) = \inf_{\Psi} d^2_\nu(\phi',\psi) = \inf_{\Psi'}d^2_\mu(\phi,\psi'),
\ee
where the first infimum is taken over parallel transports $\Psi$ of $\phi$ along geodesics joining $\mu$ toward $\nu$, the second one over parallel transports $\Psi'$ of $\psi$ along geodesics joining $\nu$ toward $\mu$ and where we use the notations $(e_1)_\#\Psi(dx,dz) = \nu(dx)\phi'_x(dz)$ and $(e_1)_\#\Psi'(dx,dz) = \mu(dx)\psi'_x(dz)$.
\end{Def/Prop}
\begin{proof}
Define 
$$
I_1 :=\inf_{\Psi} d^2_\nu(\phi',\psi), \quad I_2 := \inf_{\Psi'}d^2_\mu(\phi,\psi').
$$
We need to prove $I_1= I_2$. Let $\Psi$ be a parallel transport of $\phi$ along a geodesic $\theta$ joining $µ$ toward $\nu$. Furthermore, consider an optimal coupling between $(e_1)_\#\Psi$ and $\nu(dx)\psi_x(dz)$. By gluing the two together, we are allowed to consider random variables $(X,Z,Y,Z')$ such that $(X,Y)$ is an optimal coupling for $W_2(µ,\nu)$, $\mathcal L(X,Z) = µ(dx)\phi_x(dz), \mathcal L(Y,Z') = \nu(dx)\psi_x(dz)$ and $(Y,Z,Z')$ is optimal for $d^2_\nu(\phi',\psi)$. The cost in $I_1$ associated to $\Psi$ is thus $\mathbb E [|Z-Z'|^2]$. Note that $\mathcal L(((1-t)Y + t X, Z')_{t \in [0,1]})$ provides a parallel transport of $(\nu,\psi)$ along a geodesic joining $\nu$ toward $µ$. Hence, if we call $\Psi'$ this parallel transport, its associated cost in $I_2$ is $d^2_\mu(\phi,\psi') \leq \mathbb E[|Z-Z'|^2]$. Hence, we deduce that $I_2 \leq I_1$. By symmetry, we conclude that $I_1 = I_2$.
\end{proof}
\begin{Rem}
We could also have defined $\mathcal E $ with a supremum over geodesics, and we would have arrived at a somehow similar notion.
\end{Rem}

\subsection{A distance on $T\mathcal P_2(\R^{d})$}
There is no reason that a distance on $T\mpt$ can be derived immediately from $\mathcal E$. This type of phenomena also occurs in differential geometry. We propose here a distance on $T\mpt$ and present links with the quantity $\mathcal E$. Let $µ,\nu \in \mpt$ and $\phi \in T_µ(\mpt), \psi \in T_\nu(\mpt)$. Recall that by disintegration, considering $(µ,\phi)$ is equivalent to considering $\gamma(dx,dz) = µ(dx)\phi_x(dz) \in \mathcal P_2(\R^{2d})$. Thus, we identify the tangent sheaf with $\mathcal P_2(\R^{2d})$. Hence, measuring the distance between $(µ,\phi)$ and $(\nu,\psi)$ can be done by means of the Wasserstein distance on $\mathcal P_2(\R^{2d})$, which yields the definition of our distance:
$$
\mathcal D^2((µ,\phi),(\nu,\psi)) := W_2^2(µ(dx)\phi_x(dz),\nu(dx)\psi_x(dz)).
$$
The standard properties of the Wasserstein distance imply the following.
\begin{Prop}
The function $\mathcal D$ defines a metric on $\mathcal P_2(\R^{2d})$. Furthermore, for any $µ,\nu \in \mpt$ and $\psi \in T_µ(\mpt), \phi \in T_\nu(\mpt)$, $\mathcal D((µ,\phi),(\nu,\psi)) \geq W^2_2(µ,\nu)$.
\end{Prop}

A bound from above on $\mathcal D^2$ can be obtained as follows.
\begin{Prop}\label{prop:comppara}
Let $µ,\nu \in \mpt$, and $\phi \in T_µ(\mpt), \psi \in T_\nu(\mpt)$. Then
\be\label{eq2}
\mathcal D^2((µ,\phi),(\nu,\psi)) \leq W_2^2(µ,\nu) + \mathcal E((µ,\phi),(\nu,\psi)).
\ee
Furthermore, the inequality can be strict and we do not have in general $\mathcal E \leq \mathcal D^2$.
\end{Prop}
\begin{proof}
Consider a parallel transport of $\phi$ along a geodesic between $µ$ and $\nu$. This naturally defines a coupling between $µ(dx)\phi_x(dz)$ and $\nu(dx)\psi_x(dz)$. Hence, the left hand side is smaller than the right hand side in \eqref{eq2}. 

When $(µ,\phi) = \mathcal L((U,\lambda U))$ and $(\mu,\psi) = \mathcal L((1-U,\lambda U))$ for $U$ a uniform random variable on $[0,1]$ (hence $\nu = µ$), we can compute $\mathcal E ((µ,\phi),(µ,\psi)) = \lambda^2/3$ whereas $\mathcal D^2 ((µ,\phi),(µ,\psi)) \leq 1/3$.
\end{proof}
The previous result highlights the fact that there is a choice to be made when comparing different element in $T\mpt$. A point of view purely based on parallel transport should rely on the quantity $\mathcal E$. On the other hand, taking inspiration from differential geometry, $\mathcal D$ seems to behave like the natural geodesic distance on $T\mpt$. Outside of this section, we shall focus on the approach with parallel transport, namely because it is more natural when being interested in the regularity of functions $U : \mpt \to \R$.\\

We do not go into the full details of constructing an equivalent of the Sasaki metric on the tangent sheaf $T\mpt$ (it is a natural metric when we try to look at $T\mpt$ as a Riemannian manifold), as we do not need it for applications. Nonetheless, we briefly comment on it since $\mathcal D$ looks like what the Sasaki distance could be. Building on the structure of $T\mpt$ which can be seen as $\mathcal P(T\R^d)$ with $T\R^d = \R^{2d}$, we know that $T(T\mpt)$ can be seen as $\mathcal P(\R^{4d})$. Furthermore, the normal crossing of the vertical part of the tangent space to $T\mpt$ of the parallel transport can be stated formally quite simply, namely as the tangent part of the parallel transport is constant in this case in flat geometry. The precise decomposition of $T_\gamma T\mpt$ into an horizontal and a vertical part seems to be the most interesting question here. Clearly, by the construction we proposed for parallel transport, it is immediate to check that $(\frac{d}{dt})_\#\Psi$ is concentrated on curves $(x(t),v(t))$ on $T\R^d$ such that $v'(t)$ crosses normally the tangent space $T_{x(t)}\R^d$.\\

\subsection{Reminders: Derivatives of function of probability measures}
Let $U: \mpt \to \R$. The function $U$ is said to be differentiable at $µ \in \mpt$ if there exist a \emph{deterministic} element $\phi \in T_µ(\mpt)$ and a modulus of continuity $\omega$ such that, for any $\nu \in \mpt$, $\gamma \in \Pi(µ,\nu)$
$$
\left|U(\nu)  - U(µ) - \int_{\R^{2d}}(y-x)\cdot\phi(x)\gamma(dx,dy)\right| \leq C_2(\gamma)\omega(C_2(\gamma)),
$$
recall that $C_2(\gamma) := \left( \int_{\R^{2d}}|x-y|^2\gamma(dx,dy)\right)^\frac12$. In such a case we note $D_µ U(µ) = \phi \in L^2((\R^d,µ),\R^d)$.
\begin{Rem}
There are different notions of differentiability on $\mpt$ that can be found in the literature, see chapter 5 in \citep{carmona} for instance. Quite often, in the previous definition, the requirement is only made for coupling $\gamma$ which are optimal for $W_2(µ,\nu)$. In fact, there is an equivalence between the two notions, see Theorem 3.21 in \citep{gangbo}, and we do not know any short proof of this fact. Here, we prefer to stick to the formulation with all couplings, as it is a nice non-Hilbertian formulation of the derivative introduced by P.-L. Lions on the lift of the function. In particular, one can check that we do not use in the following the Hilbertian structure of $L^2(\Omega,\R^d)$.
\end{Rem}
\begin{Rem}
The previous definition of differentiability contains the fact that $D_µ U(µ)$ is a deterministic element of the tangent space. This fact is not obvious. It has been proven by Lions in \citep{lions} by introducing its lift on the space of random variables.
\end{Rem}

\subsection{$\mathcal C^{1,\alpha}$ regularity on $\mpt$}
Let $U: \mpt \to \R$ be a differentiable map at any $µ \in \mpt$. The aim of this section is to define a notion of continuity, H\"older or Lipschitz regularity for the map $µ \to D_µ U(µ)$. Of course, since for every $µ \in \mpt$, $D_µ U(µ) \in L^2((\R^d,µ),\R^d)$, this map is valued in $\cup_{\nu \in \mpt}L^2((\R^d,\nu),\R^d)\subset T\mpt$ and a suitable notion of parallel transport is needed to compare, for $µ\ne \nu$ $D_µ U(µ) \in L^2((\R^d,µ),\R^d)$ and $D_µ U(\nu) \in L^2((\R^d,\nu),\R^d)$. We adopt the following definition, in which we make constant use of the notation $\gamma \in \Pi(µ,\nu)$, of Proposition \ref{prop:analogchemincoupl} and of the notation $\mathcal T^\gamma(\psi)$ to denote the parallel transport of $\psi$ along the coupling $\gamma$.
\begin{Def}
Let $U: \mpt \to \R$ be a differentiable map at any $µ \in \mpt$. We say that it is $\mathcal C^1$ if, for any $µ \in \mpt$, 
$$
\lim_{C_2(\gamma) \to 0}d_\nu(\mathcal T^\gamma(D_µ U(µ)),D_µ U(\nu)) = 0.
$$
For $\alpha \in (0,1)$, we say that this function is in $\mathcal C^{1,\alpha}(\mpt)$ if 
\be\label{defI}
I_{\alpha}(D_µ U):=\sup_{\gamma, 0 <C_2(\gamma) \leq 1}\frac{d_\nu(\mathcal T^\gamma(D_µ U(µ)),D_µ U(\nu)) }{C_2(\gamma)^\alpha} < \infty.
\ee
The space $\mathcal C ^{1,1}(\mpt)$ is defined in a similar way except for the fact that the supremum is taken over all $\gamma$ such that $C_2(\gamma) > 0$.
\end{Def}
\begin{Rem}
Remark that, for $µ,\nu \in \mpt$, $\phi,\psi$ being deterministic elements of respectively $T_µ(\mpt),T_\nu(\mpt)$, and $\gamma \in \Pi(µ,\nu)$,
$$
d^2_\nu(\mathcal T^\gamma(D_µ U(µ)),D_µ U(\nu))  = \int_{\R^d\times \R^d}|\phi(x)-\psi(y)|^2\gamma(dx,dy).
$$
\end{Rem}

This definition is perfectly equivalent to standard $\mathcal C^{1,\alpha}$ regularity of the lift $\tilde U$ of $U$ in $L^2(\Omega,\R^d)$ introduced in Section \ref{sec:lift}, except that this language allows to define such a solution even when the state space is not $\R^d$. We believe it is a strong justification of the regularity of the regularization of $W_2^2$ introduced in \citep{bertucci2024approximation}.

\begin{Rem}
Below we shall comment on the choice of the quantity $C_2(\gamma)^\alpha$ in \eqref{defI}, namely that we could also have used $C_{2\alpha}(\gamma)^\alpha$.
\end{Rem}
The following holds.
\begin{Prop}
For $\alpha \in (0,1]$, $\mathcal C^{1,\alpha}(\mpt)\cap L^\infty(\mpt)$ is a normed vector space when equipped with the norm $\mathcal N$ defined by
$$
\mathcal N(U) := \|U\|_\infty + \sup_{µ \in \mpt}\{d_µ(D_µU(µ),0)\} + I_{\alpha}(D_µU).
$$
\end{Prop}
\begin{Rem}
\begin{itemize}
\item The same result holds for $\mathcal C ^{1,0}(\mpt)$ with the convention that $I_0 (D_µU) = 0$ for any $U$.
\item The boundedness of $U$ and $D_µU$ can be relaxed in the usual way: by considering weighted norms which allow for a fixed growth in $W_2(µ,\delta_0)$ (polynomial e.g.).
\end{itemize}
\end{Rem}
\begin{proof}
The probabilistic reformulation of $d_µ$ implies that $I_\alpha$ indeed satisfies $I_\alpha(\lambda D_µ) = |\lambda|I_\alpha(D_µ U)$. Hence, to verify that $\mathcal N$ is a norm, we only need to check the triangular inequality. This does not raise any issue, for instance it follows easily from the probabilistic representation of $d^2_µ$. The fact that $\mathcal N$ is a norm then implies the vector space structure.

Hence, it only remains to verify the completeness. Let $(U_n)_{n\geq 0}$ be a Cauchy sequence for $\mathcal N$. Since it is one for the $\|\cdot\|_\infty$ norm, we know that it converges uniformly toward some function $U$, which is continuous. To prove the convergence of the derivative, consider the map $\phi_n : µ \mapsto µ(dx)\delta_{D_µU_n(µ)(x)}(dz)$, valued in $\mathcal P_2(\R^{2d})$. Remark that for any $µ \in \mpt$, $(\phi_n(µ))_{n \geq 0}$ is a Cauchy sequence. Indeed, since we can write $d_µ(D_µU_n(µ) - D_µU_m(µ),0) = W_2(\phi_n,\phi_m)$. Hence, $(\phi_n(µ))_{n \geq 0}$ converges to some $g(µ) \in \mathcal P_2(\R^{2d})$ with first marginal $µ$. Furthermore, this convergence is uniform in $µ$. We now need to check that $g = D_µU$. Since, for all $n \geq 0$, $U_n$ is $\mathcal C^1$, we obtain that for any $\gamma \in \Pi(µ,\nu)$ for $µ,\nu \in \mpt$,
$$
U_n(\nu) - U_n(µ) = \int_0^1\int_{\R^{2d}}D_µU_n(((1-t)\pi_1+t \pi_2)_\#\gamma)(x)(y-x)\gamma(dx,dy)dt.
$$
Passing to the limit in this expression, and denoting $\gamma(dx,dy) = µ(dx)k_x(dy)$, we deduce that
$$
U(\nu) - U(µ) = \int_0^1\int_{\R^{3d}}z\cdot(y-x)g(((1-t)\pi_1+t \pi_2)_\#\gamma)(dx,dz)k_x(dy)dt.
$$
From this, we deduce that, if we denote $g(µ)(dx,dz) =µ(dx)\psi_{µ,x}(dz)$, then $x \mapsto \int_{\R^d}z\psi_{µ,x}(dz) = D_µU(µ)$. In fact, we can show that $g(µ)$ is in fact of the form $(Id,\phi)_\#µ$ for some map $\phi$, which is then nothing else than $D_µU(µ)$. In this setting, this is a consequence of Theorem 3.17 in \citep{gangbo} (it is even a consequence of a weaker results of Lions in \citep{lions} under our assumptions). Hence, we obtained that $U$ is differentiable everywhere. Passing to the limit in the quotient which defined $I_\alpha(U_n)$, we easily obtained that $I_\alpha(U) < \infty$, hence, the result is proved.

\end{proof}

We also recall Proposition 5.36, vol I in \citep{carmona}, stated with other notation, but which here takes the form.
\begin{Prop}\label{prop:reglip}
Let $U: \mpt \to \R$ be such that $I_1(D_µU) < \infty$. Then, for any $µ$, there exists a $I_1(D_µU)$ Lipschitz function $\phi_µ: \R^d \to \R^d$ such that $D_µ U = \phi_µ$ $µ$ almost everywhere.
\end{Prop}
We do not reproduce its rather heavy proof and we leave the reader to check that its proof does not use the Hilbert structure of $L^2(\Omega,\R^d)$, and that it can be thus extended instantly to the case of a function $U: \mpo \to \R$ for $\mo$ a smooth domain.

\subsection{Discussion on the $\mathcal C^{1,\alpha}$ norm and its consequences}
There are mainly two points that we want to insist upon in the definition of $I_\alpha$ in \eqref{defI}.\\

 The first one is the fact that the supremum is taken over all couplings and not only over optimal couplings. As we mentioned above, the same choice/problem arises when defining the notion of derivative, and it has be shown in \citep{gangbo} that the two notions are equivalent. Of course, for $I_\alpha$, one implies the other, but it is not clear for us to show an equivalence between the two notions and we leave this question for future research. In particular, the proof of \citep{gangbo} uses in a crucial manner the fact that the term we want to estimate can be seen as a part of a Taylor expansion, which is not true here.\\

The second point we want to insist upon is that, when $\alpha \in (0,1)$ the choice to divide by $C_2(\gamma)^\alpha$ might not be the only natural one. Indeed, one could have also defined 
$$
J_\alpha(D_µU):= \sup_{\gamma, 0 < C_2(\gamma)< 1}\frac{d_\nu(\mathcal T^\gamma (D_µ U(µ)),D_µ U(\nu))}{C_{2\alpha}(\gamma)^\alpha}.
$$
Since, by Jensen's inequality $C_{2\alpha}(\gamma)^\alpha \leq C_2(\gamma)^\alpha$, $J_\alpha(D_µU)\geq I_\alpha(D_µU)$. Furthermore, the proof of Proposition 5.36, vol I in \citep{carmona} can be carried out in a similar manner to obtain the following.
\begin{Prop}
Let $U: \mpt \to \R$ be such that $J_\alpha(D_µU) < \infty$. Then, for any $µ$, there exists a $\mathcal C^\alpha(\R^d,\R^d)$ continuous function $\phi_µ$, such that $\|\phi_µ\|_{\mathcal C^\alpha} \leq J_\alpha(D_µU)$ and $D_µ U = \phi_µ$ $µ$ almost everywhere.
\end{Prop}
\begin{proof}
We do not reproduce this quite long and tedious proof. We mention that following the exact same argument as in \citep{carmona}, we arrive at the required result by, instead of doing the Taylor expansion around the element of interest $z_0 \in \R^d$ (with the notation of \citep{carmona}), we instead bound the quotient which appears in the expression of the $\mathcal C^\alpha$ norm of $D_µU$.
\end{proof}
Furthermore, we were not able to construct a function $U: \mpt \to \R$ such that $J_\alpha(D_µU) = \infty$ while $I_\alpha(D_µ U) < \infty$, also not knowing if such a function exists. We also leave these questions for future research.

\subsection{Extensions to other cases than $\mpt$}
The results and definitions above translate immediately if, instead of working on $\mpt$, we are interested in either $\mpp$ or $\mathcal P(\mo)$ for a smooth bounded domain $\mo \subset \R^d$. The case of probability measures over a smooth manifold requires a bit more care, as we constantly made use of Proposition \ref{prop:analogchemincoupl} in the previous, which of course has no reason to be true in this setting. Thus, when defining the equivalent of $I_\alpha$ in this setting, one has to be careful with the fact that the supremum has to be taken over all Lagrangian path $\eta$ as follows.
$$
\sup_{\eta \in LP(µ,\nu)}\frac{d_\nu(\mathcal T^\eta(D_µU(µ)),D_µU(\nu))}{C_2(\eta)^\alpha},
$$ 
where 
$$
LP(µ,\nu) := \{ \eta \in \mathcal P(\mathcal C([0,1],M)), (e_0)_\#\eta = µ, (e_1)_\#\eta = \nu, \eta \text{ is concentrated on } C^1([0,1],M)\},
$$
and, for $\eta \in LP(µ,\nu)$,
$$
C_2(\eta) := \left(\int_{ \mathcal C([0,1],M)}\int_0^1|y'(t)|_{y(t)}^2dt \eta(dy)\right)^\frac12.
$$
Except for this change, the previous sections can be extended quite straightforwardly.

\section{Translated curves and application in optimal control problems}\label{sec:trans}
In this section, we explain and prove how ideas introduced above yields quite simply regularity of the value function of some optimal control problems in the space of probability measures. Once again, we focus here on the case of $\mpt$. We comment on several extensions at the end of the section. The core idea of this section is to translate curves in $\mpt$. We do not claim we shall be able to translate any curves so we start by explaining in details the curves with which we are going to be interested in. Then, we explain how to translate them, and finally how to use this translations in optimal control problems.\\

Given a curve $(m_t)_{t \in [0,1]}$ in $\mpt$  we would like to translate it to another curve $(\tilde m(t))_{t \in [0,1]}$. By translation, we mean that if the two curves are sufficiently smooth to have time derivative, then, in some sense, $\tilde m'(t) \in T_{\tilde m(t)}$ should be a parallel transport of  $m'(t) \in T_{m(t)}(\mpt)$ along a given path between $m(t)$ and $\tilde m(t)$. More precisely, since we are motivated by control problems, our fundamental objective is in fact the following: given a curve $(m(t))_{t \in [0,1]}$ generated by a control $(\alpha_t)_{t \in [0,1]}$ (typically $\alpha_t \in T_{m(t)}(\mpt)$), can we generate another curve $(\tilde m(t))_{t \in [0,1]}$ based on another control $(\beta_t)_{t \in [0,1]}$ which is given so that for all $t$, $\beta_t$ is a parallel transport of $\alpha_t$ along a certain path between $m(t)$ and $\tilde m(t)$ ?

\subsection{Curves that we want to translate}\label{sec:curves}

We focus first on the construction of a curve given a control. When restricting our attention to controls which are deterministic elements of the tangent space, the question is well understood and we recall some facts about it here.

 Consider a curve $(m_t)_{t \in [0,1]}$ which is a weak solution of
\be\label{ce}
\partial_t m + \text{div}_x (\alpha m) = 0 \text{ in } (0,1)\times \R^d,
\ee
for some measurable vector field $\alpha : [0,1]\times \R^d\to \R^d$. By weak solution, we mean that we can test the previous relation against $\mathcal C^{1,c}([0,1]\times \R^d \to \R^d)$ test functions and obtain the usual relations. Assume that
\be\label{growth}
\int_0^1\|\alpha(s,\cdot)\|_{L^2(m_s)}^2ds < \infty.
\ee
Condition \eqref{growth} is far from sufficient to guarantee uniqueness of solutions $m$ of \eqref{ce}, nonetheless, it is sufficient to obtain absolute continuity of the curve, and thus a precise sense in which we can understand $m'(t) \in T_{m(t)}(\mpt)$. It is somehow summarized in Ambrosio's superposition principle that we recall here.
\begin{Theorem}[Thm 8.2.1 in \citep{ags}]\label{thm:ambrosio}
If $(m_t)_{t \in [0,1]}$ is a curve in $\mpt$ which is a weak solution of \eqref{ce} and if \eqref{growth} is satisfied, then there exists a process $(X_t)_{t \in [0,1]}$ such that $m_t = \mathcal L(X_t)$ for all $t$ and almost surely
\be\label{eds}
\forall t \in [0,1], X_t = X_0 + \int_0^t\alpha(s,X_s)ds.
\ee
Conversely, if $(X_t)_{t \in [0,1]}$ is a solution of \eqref{eds} and $\int_0^1\mathbb{E}[|\alpha(s,X_s)|^2]ds < \infty$, then $(\mathcal{L}(X_t))_{t \in [0,1]}$ is a weak solution of \eqref{ce}.
\end{Theorem}
Since there is no uniqueness of solutions $m$ of \eqref{ce} under the sole condition \eqref{growth}, this approach does not allow to talk about the evolution $m$ of associated to a control $\alpha$ but simply about an evolution $m$ associated to the control $\alpha$.\\

Since we saw that the parallel transport of a deterministic element of the tangent space may fail to be deterministic, we want to consider an analogous of the previous approach for general elements of the tangent space.\\

This objective puts us into the territory of a quite usual problem in stochastic control: how to generate sufficient noise ? In other words, if we consider general controls $\psi: [0,T]\times \R^d \to \mathcal P(\R^d)$, can we consider a natural associated evolution $(m_t(dx)\psi_{t,x}(dz))_{t \in [0,1]}$ such that
$$
\int_0^1\int_{\R^d\times \R^d}|z|^2\psi_{s,x}(dz)m_s(dx)ds < \infty \quad ?
$$

The answer to this question is no in general. Consider for instance an hypothetical case in which, for all $t \geq 0$, $m_t(dx)\psi_{t,x}(dz) = m_t(dx)\rho(dz)$ for $\rho$ the normal distribution. This would correspond to a situation in which, at any time, we are able to randomize the speed of the particles, independently of their current position, hence of what has been done before. The natural such candidate would be given by a process $(X_t,Z_t)_{t \in [0,1]}$ where $\mathcal L((X_t,Z_t)) = m_t(dx)\rho(dz)$ and $X_t = X_0 + \int_0^tZ_sds$. We do not know how to construct such a process. On the other hand, we can characterize situations which we can represent as the next result shows.
\begin{Prop}\label{prop:weakambrosio}
Let $(X_t,Z_t)_{t \in [0,1]}$ be a $\R^d\times \R^d$ process such that, 
$$
\int_0^1\mathbb E[|Z_t|^2] dt < \infty
$$
 and, almost surely, for all $t \in [0,1]$
\be\label{cond16}
X_t = X_0 + \int_0^tZ_sds.
\ee
Then, there exists a probabilistic space $K \subset \R^N$, a measurable map $b: [0,T]\times \R^d \times K \to \R^d$ and a continuous path $(µ_t)_{t \in [0,1]}$, valued in $\mathcal P(\R^d\times K)$ such that 
\be\label{cond15}
((\pi_1,b(t,\pi_1,\pi_2)_\#µ_t)_{t \in [0,1]} = (\mathcal L(X_t,Z_t)_t)_{t \in [0,1]}
\ee
 and $(µ,b)$ is a weak solution of the continuity equation
\be\label{ceK}
\partial_t µ + \text{div}_x(bµ) = 0 \text{ in } (0,1)\times \R^d\times K,
\ee
with initial condition $µ_0$ such that $(\pi_1)_\#µ_0 = \mathcal L(X_0)$.
Conversely, given a probabilistic set $(K\subset \R^N,\mathcal B,\mathbb P')$ and a solution $(µ,b)$ of \eqref{ceK} satisfying 
$$
\int_0^1\int_{\R^d\times K}|b(t,x,k)|^2µ_t(dx,dk)dt < \infty,
$$
there exists a process $(X_t,Z_t)_{t \in [0,1]}$ such that \eqref{cond15} and \eqref{cond16} hold.
\end{Prop}
\begin{proof}
The proof is rather simple. On the first side, we note $(\Omega,\mathcal A,\mathbb P)$ the atomeless probabilistic space onto which the process $(X_t,Z_t)_{t \in [0,1]}$ is defined. Without loss of generality, we can assume that $(\Omega,\mathcal A,\mathbb P)= ([0,1]^N,\mathcal B, Leb)$. Taking $K = \Omega$, $b(t,x,k) = Z_t(k)$ for $(t,x,k) \in [0,1]\times \mo\times K$ and $µ_t = \mathcal L(X_t,Z_t)$ yields the solution.

On the other side, we use Ambrosio's superposition principle (in $\R^d\times \R^N$) to obtain a process $(X_t,k_t)_{t \in [0,1]}$ from $µ$ on a probabilistic space $(\Omega,\mathcal A,\mathbb P)$. Observe that almost surely $k$ is constant. We then define $Z_t(\omega) = b(t,X_t(\omega),k_t(\omega))$ and the result follows.
\end{proof}
\begin{Rem}
Note that the set $K$ is somehow free in the previous statement, as it can be changed into any non-atomic probabilistic space.
\end{Rem}
\begin{Rem}\label{rem:conditionalexpectation}
Clearly, the curves satisfying the assumptions of Theorem \ref{thm:ambrosio} also satisfy the assumptions of Proposition \ref{prop:weakambrosio}. Furthermore, consider $(b,µ,K)$ as in Proposition \ref{prop:weakambrosio}. Applying the more general version of Theorem \ref{thm:ambrosio} to the process given by Proposition \ref{prop:weakambrosio}, we obtain that there exists a measurable $\alpha:[0,T]\times \R^d\to \R^d$ such that $m_t :=(\pi_1)_\#µ_t$ solves \eqref{ce}. In fact,  $\alpha(t,\cdot)$ is given $m_t$ almost everywhere by $\alpha(t,x):= \mathbb{E}[Z_t|X_t =x]$.
\end{Rem}
The interpretation of this reformulation is that the set $K$ is the set used to randomize the speed we associated to particles located at $x \in \R^d$. Ambrosio's superposition principle basically states that we can forget this space if we are only interested in the evolution of the marginal in $x \in \R^d$ of \eqref{ceK}. Since it is not the case here, this generalization shall have some importance.\\

The main role of the previous result is to avoid pathological situation such as the one mentioned before the Proposition.

\subsection{Translating the curves}

Given a $\mathcal C^1$ curve $(x_t)_{t \in [0,1]}$ in $\R^d$, and a point $y_0 \in \R^d$, we can easily understand the translation of $(x_t)_{t \in [0,1]}$ to $y_0$ as the curve $(y_t)_{t \in [0,1]}$ defined by $y_t:= y_0 + \int_0^t\dot x_sds$. We now define a similar notion for curves on $\mpt$ satisfying the assumptions of Proposition \ref{prop:weakambrosio}. Ideally, we would like to say that we simply need, at each time, to parallel transport the "speed" of the curve toward another point. However, this strategy runs into the following difficulty: a proper parallel transport is not simply defined between two measures, but along a path (or a minima a coupling in this case). Thus, we need to treat with care the way we are going to connect our two curves as time passes. We shall do so by not only seeking for the translated curve $(\tilde m_t)_{t \in [0,1]}$ of $( m_t)_{t \in [0,1]}$, but rather by seeking for curves which are couplings between them.

We now propose the following notion of translated curve.
\begin{Def}
For any curve $(m_t)_{t \in [0,1]}$ which can be represented through $b,K,µ$ as in Proposition \ref{prop:weakambrosio}, any $\tilde m_0 \in \mpt$ and coupling $\gamma_0 \in \Pi(m_0,\tilde m_0)$, a coupling curve $(\eta_t)_{t \in [0,1]}$ is a curve in $\mathcal P_2(\R^d\times \R^d\times K)$, weak solution of
$$
\partial_t \eta + \text{div}_x(b \eta) + \text{div}_y(b \eta) = 0 \text{ in } (0,1)\times \R^d\times \R^d\times K,
$$
with initial condition $\eta_0(dx,dy,dk) = \gamma_0(dx,dy)\psi_x(dk)$ where $m_0(dx)\psi_x(dk) := µ_0(dx,dk)$. The curve $((\pi_2)_\#\eta_t)_{t \in [0,1]}$ is called a translation of $(m_t)_{t \in [0,1]}$ along $\gamma_0$. Furthermore, the curve $((\pi_1,\pi_2)_\#\eta_t)_{t \in [0,1]}$ is called the associated translated pair.
\end{Def}
\begin{Rem}
Of course, if $(m_t)_{t \in [0,1]}$ is represented (as it is usually the case), without a probabilistic space $K$, then it can be omitted in the previous. Note however that in such a case, it is still likely that the base space $\R^d$ shall play the role of $K$ for the translation $((\pi_2)_\#\eta_t)_{t \in [0,1]}$.
\end{Rem}

Coupling curves exist in general as the next result shows.
\begin{Prop}\label{Prop:deftrans}
Let $(m_t)_{t \in [0,1]}$ be a curve which can be represented through $b,K,µ$ as in Proposition \ref{prop:weakambrosio}. For any $\gamma_0\in \mathcal{P}(\R^{2d})$ with $(\pi_1)_\#\gamma = m_0$, there exists an associated coupling curve, translating $(m_t)_{t \in [0,1]}$ along $\gamma_0$. If $(\gamma_t)_{t \in [0,1]}$ is the translated pair, then $C_2(\gamma_t)$ does not depend on time. More precisely, if $(X_t,Y_t)_{t \in [0,1]}$ is a process of law $(\gamma_t)_{t \in [0,1]}$. Then, almost surely, $X_t-Y_t$ is constant in time.

\end{Prop}
\begin{proof}
By the gluing Lemma, we can consider both a process $(X_t(\omega),k(\omega))_{t \in[0,1]}$ of marginal law $(µ_t)_{t \in [0,1]}$ on a probabilistic space $(\Omega,\mathcal A, \mathbb P)$ such that, almost surely, $X_t= \int_0^tb(s,X_s,k)ds + X_0$, and a random variable $Y_0$ such that $(X_0,Y_0)\sim\gamma_0$. Define the process $Y_t = Y_0 + \int_0^tb(s,X_s,k)ds$. Setting $\eta_t := \mathcal L(X_t,Y_t,k)$ completes the proof. The rest follows easily.
\end{proof}

Note that there is no uniqueness of such a procedure as the next example shows. This is mainly due to similar reasons as the non-uniqueness of parallel transport on $\mpt$.

\begin{Ex}
For all $t \in [0,1]$, define $m_t = \frac12(\delta_t + \delta_{-t})$. Consider $\gamma_0(dx,dy) = \frac12\delta_0(dx)(\delta_{-2}(dy) + \delta_2(dy))$. Remark that $(m_t)_{t \in [0,1]}$ is a solution of \eqref{ce} with $\alpha(t,x) = 2\mathbb{1}_{x \geq 0} -1$. The three following trajectories are all translations of $(m_t)_{t \in [0,1]}$ along $\gamma_0$
$$
\eta^1_t = \frac12(\delta_{-2-t} + \delta_{2 + t});\quad 
 \eta^2_t = \frac12 (\delta_{-2+t} + \delta_{2 - t});\quad 
\eta^3_t= \frac 14(\delta_{-2-t} + \delta_{2 + t} + \delta_{-2+t} + \delta_{2 - t}).
$$

\end{Ex}

As in the gluing Lemma, in certain "deterministic" situations, we are able to obtain uniqueness of the translation in the sense that all possible translations yields the same translated curve.
\begin{Prop}
Assume that $(m_t)_{t \in [0,1]}$ is a weak solution of \eqref{ce} for some $\alpha \in L^1((0,1),W^{1,\infty}(\R^d))$ and consider $\tilde m_0 \in \mpt$, $\gamma_0\in \Pi(m_0,\tilde m_0)$. Then, there is a unique translating coupling of $(m_t)_{t \in [0,1]}$ along $\gamma_0$.
\end{Prop}
\begin{proof}
The result is trivial, since, by standard Cauchy Lipschitz theory, the flow of the ODE is uniquely defined, there is thus no room for any non-uniqueness.
\end{proof}
\begin{Rem}
Note that we have here only an assumption on the regularity of the vector field, compared to the "or" type of assumption of Proposition \ref{prop:uniqpara}. This is due to the time dependent nature of the present problem. We leave as a simple exercise to the interested reader to establish that deterministic initial coupling is not sufficient, if the velocity field is not smooth, for instance by considering the case of the curve $(m_t)_{t \in [0,1]}$ given by $m_t = \frac 12(\delta_{f(t)} + \delta_{-f(t)})$ where $f(t) = \max(-\frac 12 -t; 0; -\frac 12 + t)$.
\end{Rem}

This translation enjoys several other properties.

\begin{Prop}
The following hold.
\begin{enumerate}[(i)]
\item Let $(x_t)_{ t \in [0,1]}$ be a $\mathcal C^1$ curve on $\R^d$, $\gamma(dx,dy) = \delta_{x_0}(dx)\delta_{y_0}(dy)$ for some $y_0 \in \R^d$. Then the translation of $(\delta_{x_t})_{t \in [0,1]}$ along $\gamma$ is given by $(\delta_{y_t})_{t \in [0,1]}$. 
\item Let $(m_t)_{t \in [0,1]}$ be a curve as in Proposition \ref{Prop:deftrans} and $\gamma$ some optimal coupling for $W_2$ whose first marginal is $m_0$. Let $\tilde m_t$ be a translation of $(m_t)_{t \in [0,1]}$ along $\gamma$. Then $t \mapsto W_2(m_t,\tilde m_t)$ is bounded by $C_2(\gamma)$.
\end{enumerate}
\end{Prop}
\begin{proof}
For the first point, it suffices to remark that there is a unique possible translated curve. The second point follows directly from point $(v)$ in Proposition \ref{Prop:deftrans}.
\end{proof}
Note that there are cases in which the $W_2$ distance between a curve and its translation actually decreases before increasing again, as the previous example showed. A less pathological example is the following.
\begin{Ex}
Let $m_t := \frac12(\delta_{-4+3t} + \delta_{4-3t})$. It is clearly a solution of \eqref{ce} for some smooth $\alpha$, constant in time, such that $\alpha (x) = -3$ if $x \geq 1$ and $\alpha(x) = 3$ if $x \leq -1$. Let $\gamma(dx,dy) = \frac12(\delta_{-4}(dx)\delta_{-1}(dy) + \delta_4(dx)\delta_1(dy))$. It is clearly an optimal coupling for $W_2$ and one can check that $\tilde m_t := \frac12(\delta_{-1 + 3t} + \delta_{1-3t})$ is the unique translation of $(m_t)_{t \in [0,1]}$ along $\gamma$. Then 
$$
W^2_2(m_t,\tilde m_t) = \begin{cases} 9 \text{ if } t \leq \frac 13,\\ (5 - 6t)^2 \text{ else.} \end{cases}
$$
\end{Ex}
This last example is yet another illustration (if it was still needed at this point) that the geometry of $(\mpt,W_2)$ is not flat at all...

\subsection{Applications in optimal control}

Consider $T > 0$, $L: [0,T]\times \R^d\times \R^d \times \mpt \to \R_+$, $G : \mpt \to \R_+$ and the optimal control problem defined for all $t \in [0,T],\bar m\in \mpt$
$$
V(t,\bar m) := \inf_{(\alpha,m)} \left\{\int_t^T \int_{\R^{d}}L(s,x,\alpha(s,x),m_s)m_s(dx)ds + G(m_T)\right\},
$$
where the infimum is taken over all pairs $(\alpha,m)$ such that $m: [t,T] \mapsto \mpt$ is continuous, $\alpha: [0,T]\times\R^d \to \mathcal P(\R^d)$ is measurable and they solve \eqref{ce} with $m_t= \bar m$. The function $V$ is called the value function of the problem. Typically, some growth condition is assumed on $L$ in its third argument, so that $\alpha$ has to exhibit certain bounds. Quite often, those bounds are that for almost every $s \in [t,T]$, $\alpha \in L^2((\R^d,m_s),\R^d)$. 

The interpretation of such a problem is that one can control the measure $m$ through the continuity equation by choosing the speed $\alpha$. This interpretation has to be manipulated with care since, because no particular regularity condition on $\alpha$ has been imposed, the continuity is not necessary uniquely solvable, thus the presence of $m$ in the infimum as well (since, given an $\alpha$, multiple $m$ can be solutions). We refer to \citep{bertuccibook} for more details on this kind of problem.\\

In view of \citep{bertucci2024stochastic} and the notion of parallel transport above, it may seem arbitrary to restrict ourselves to $\alpha$ which are deterministic elements of the tangent space. Hence it is also natural to consider the following problem
\be\label{defUK}
U(t,m) :=  \inf_{(b,µ,K)} \left\{\int_t^T \int_{\R^{d}\times K}L(s,x,b(s,x,k),(\pi_1)_\#µ_s)µ_s(dx,dk)ds + G(m_T)\right\},
\ee
where the infimum is taken over all $(b,µ,K)$ such that \eqref{ceK} is satisfied and such that $(\pi_1)_\#µ_0 = m$.

We always have $U \leq V$, and when $L$ is convex in its third argument, we also have that $V= U$ because of Jensen's inequality and Remark \ref{rem:conditionalexpectation}. We focus in the rest of this section on $U$ as we believe it is a more natural object, for essentially all the arguments we gave when choosing our tangent space above. Our main result on $U$ is the following.
\begin{Theorem}\label{thm:lipschitz}
Assume that there exists $C > 0$ such that $L$ and $G$ satisfy
$$
\ba
&\|G\|_\infty \leq C, \quad \forall µ,\nu, \quad |G(m)-G(m')| \leq CW_2(m,m'),\\
&\forall s,x,z,m, \quad C^{-1}|z|^2 - C \leq L(s,x,z,m) ,\\
&\forall s,x,y,z,m,m', \quad |L(s,x,z,m) - L(s,x,z,m')| \leq C(1 + |z|)(|x-y| + W_2(m,m'))
\ea
$$
Then, there exists $C_0$ depending only on $C$ and $T$ such that for all $t\in [0,T],m,m' \in \mpt$,
$$
|U(t,m) - U(t,m')| \leq C_0W_2(m,m').
$$
\end{Theorem}
\begin{proof}
Let $t \in [0,T], m,m' \in \mpt$ and $\gamma$ optimal for $W_2(m,m')$. The proof consists in taking an almost optimal curve for $U(t,m)$ and translating it along $\gamma$. Using this translation for $U(t,m')$ shall then yield the result.\\

Note first that $U(t,m)$ is indeed well defined since we can simply choose the pair $(b,µ,K)$ such that $K$ is trivial, $b \equiv 0$ and $µ_s(dx) = m(dx)$. Let $\eps > 0$ and an $\eps$ optimal control $(b,µ,K)$ for $U(t,µ)$. Consider a translating coupling $\eta$ of $(b,µ,K)$ along $\gamma$. Note that $m_s = (\pi_1)_\#\eta_s$ and define $m'_s= (\pi_2)_\#\eta_s$. By $\eps$ optimality in $U(t,m)$ and admissibility in $U(t,m')$, we obtain the following bound.
$$
\ba
U&(t,m') - U(t,m) \leq  G(\tilde m_T) - G(m_T) + \eps +\\
&+\int_t^T\int_{\R^d\times \R^d\times K}L(s,y,b(s,x,k),m'_s) - L(s,x,b(s,x,k),m_s)\eta_s(dx,dy,dk) ds\\
\leq& \int_t^T\int_{\R^d\times \R^d\times K}C(1 +|b(s,x,k)|)(|x-y| + W_2( m_s,m'_s)) \eta_s(dx,dy,dk)ds\\
 &+ CW_2( m'_T,m_T) + \eps.
\ea
$$ 
Remark now that because of the growth condition on $L$, both $$I:=\int_t^T \int_{\R^d\times \R^d\times K}|b(s,x,k)|^2\eta_s(dx,dy,dk)ds <\infty$$ and $I \leq C( C+U(t,µ))$. Thus, using Cauchy-Schwarz inequality, and recalling that for any $s \in [t,T], W^2_2( m_s,m'_s) \leq \mathbb E [|X_s-Y_s|^2] = \mathbb E [|X_t-Y_t|^2] = W_2^2(m,m')$, we obtain the required result.
\end{proof}
The previous result was already known, and it was namely proven by controllability argument \citep{bertucci2024stochastic}. This proof makes transparent the link with standard proofs in finite dimensional cases, see for instance \citep{fleming}. 

\begin{Rem}
Our definition of $U$ is a sort of randomization of the controls allowed in $V$, it can be analyzed in regards of what is called relaxed controls, or mixed strategies in game theory. It could have been formulated in terms of random variables, similarly to problem of optimal control of Mc-Kean Vlasov dynamics, namely by defining $U$ as
$$
U(t,m) := \inf_{(X,Z)} \left\{ \int_t^T\mathbb E[L(s,X_s,Z_s,\mathcal L(X_s))]ds + G(\mathcal L(X_T))\right\},
$$ 
where the infimum is taken over $(X_s,Z_s)_{s \in [t,T]}$ such that $\mathcal L(X_t) = m$ and, almost surely, $X_s = X_t + \int_t^s Z_udu$. Thanks to Proposition \ref{prop:weakambrosio}, the two definitions are equivalent. We preferred the definition \eqref{defUK} as it makes transparent the fact that $K$ is chosen by the controller. In the case of two players extensions of such problems, this has an importance, as the players can both randomize their strategies, but not on the same probabilistic space \citep{bertuccitoappear}.
\end{Rem}

\subsection{Extensions}
We comment on several extensions of the previous method, namely to highlight that even if our formulation with stochastic processes is helpful, it does not rely by any mean on the Hilbertian structure of $L^2(\Omega, \R^d)$.

\subsubsection{More general controls than directly the speed}
In several practical situations, the control $\alpha$ does not affect directly the continuity equation as in \eqref{ce}, but rather intervene through another function $\beta$ as in
$$
\partial_t m + \text{div}_x(\beta(x,\alpha)m)= 0 \text{ in } (0,1)\times \R^d.
$$
This does not play a strong role in the previous. Using our method our translation on the process $(X_s,Z_s)$ where this time $z$ does not represent the speed directly but rather the control, we are still able to transport the control to another initial condition. Note that such a transport is then not the one of parallel transport on $T\mpt$, but rather a similar transport on a sheaf (see Section \ref{sec:sheaf}) which is not the tangent one. Under standard smoothness assumption on $\beta$ (globally Lipschitz for instance), we are then able to recover similar estimates.

\subsubsection{Other integrability conditions than $\mpt$}
The restriction to $2nd$ order moments plays no fundamental role in the previous, and similar results could be obtained for $\mathcal P_p(\R^d)$ with $p \ne 2$.

\subsubsection{Other cases than $\mpt$}
Cases in which we are interested in probability measures over a subdomain of $\R^d$ works exactly in the same way, formally, except for the technical treatment of the boundary of the domain.

When dealing with a smooth manifold, a similar approach can also yield some results, although several difficulties may appear. In this case, the construction of translated curves is more involved, and we need to use the appropriate notion of parallel transport on the manifold. We briefly explain the main challenge at the level of stochastic processes, restricting also our attention to optimal couplings for the sake of clarity. Consider thus a smooth bounded manifold $M$ of dimension $d$, a $M$-valued stochastic process $(X_t)_{t \in [0,1]}$ and a $\R^d$ valued process $(Z_t)_{t \in [0,1]}$ which is interpreted as the being in $T_{X_t}M$ for almost every $t, \omega$ and such that, almost surely,
\[
\frac{dX_t}{dt} = Z_t.
\]
Let $Y_0$ be a $M$-valued random variable, and consider $\eta \in \mathcal P(C([0,1],M))$, concentrated on geodesics of $M$, such that $(e_0)_\#\eta = \mathcal L(X_0)$ and $(e_1)_\#\eta = \mathcal L(Y_0)$. A translation of $(\mathcal L(X_t))_{(t \in [0,1]}$ along $\eta$ should then naturally be given by $(Y_t)_{t \in [0,1]}$, defined as being almost surely a solution of
\be\label{last}
\frac{dY_t}{dt} = \mathcal T^{\vartheta_t}(Z_t),
\ee
with initial condition $Y_0$, where $\mathcal T^{\theta}(z)$ denotes the parallel transport of $z$ along the path $\theta$ and $(\vartheta_t)_{t \in [0,1]}$ is random process taking value in the set of geodesics of $M$, such that almost surely, for almost every $t\in [0,1]$, $\vartheta_t(\omega)$ connects $X_t(\omega)$ to $Y_t(\omega)$. The main issue is that, there is no canonical way to choose $\vartheta$, even with the knowledge of $\eta$, contrary to the flat case. The possible multiplicity of geodesics then implies that the parallel transport of $Z_t(\omega)$ along these geodesics is not a Lipschitz function of $Y_t(\omega)$, which thus makes more challenging the resolution of \eqref{last}.

We believe there are two main ways to tackle this problem. The first one is to prove existence of weaker notion of solution to \eqref{last}, just as Theorem \ref{thm:ambrosio} provides a probabilistic solution to a singular ODE. The second one is to restrict our attention to cases in which $X_t(\omega)$ and $Y_t(\omega)$ are sufficiently close. Note that to prove Lipschitz estimates as in Theorem \ref{thm:lipschitz}, we are mainly interested in $µ,\nu$ such that $W_2(µ,\nu)$ is small, and, as is the distance gets smaller and smaller, the probability that a unique geodesic connects $X_t(\omega)$ and $Y_t(\omega)$ get closer to one as the points are getting closer to one another with larger and larger probability.

In any case, we leave this question for future research.

\section*{Conclusions}
We have presented another definition of the tangent space $T_µ(\mpt)$ than the one which is usually considered. It allowed us sufficient flexibility to tackle the different problems we were interested in. It seems that the only price to pay for such a generalization is the loss of uniqueness of several concepts (parallel transports or translated curves namely). We provided examples of non-uniqueness which highlight the fact that this non-uniqueness arise more from the nature of the space of probability measure than by something that we would have missed when defining $T_µ(\mpt)$.

Several times, we made an crucial use of working at the level of $\mathcal P(\mathcal C([0,1],\R^d))$ rather than at the one of $\mathcal C([0,1],\mathcal P(\R^d))$, therefore adopting a Lagrangian point of view rather than an Eulerian one. This strategy, which is aligned with the lift on the space of random variables, seems to be quite fruitful. Without trying to diminish the obvious advantages of the Eulerian formulation (conciseness, numerical methods,...), we believe that for theoretical purposes, the Lagrangian formulation is quite often more helpful.

\section*{Acknowledgments}
The author is thankful to Luigi Ambrosio, Averil Aussedat, Jean-Michel Lasry and Pierre-Louis Lions for helpful discussions and/or comments on the manuscript. The author also acknowledges a partial support from the Chair FDD (ILB) and the Lagrange Mathematics and Computing Research Center. This work was partially funded by the ERC project PaDiESeM.

\bibliographystyle{plainnat}
\bibliography{bibtan}

\end{document}